\documentclass[letterpaper, 10 pt, conference]{ieeeconf}
\IEEEoverridecommandlockouts  % This command is only needed if you want to use the \thanks command

\usepackage{fix-cm}
\usepackage{etex}

\usepackage{dblfloatfix}

\usepackage{nag}

\makeatletter
\@ifpackageloaded{xcolor}{}{%
\usepackage[table,x11names,dvipsnames,svgnames]{xcolor}%
}
\makeatother

\usepackage{colortbl}

\usepackage{graphicx}
\usepackage{wrapfig}

\definecolorset{RGB}{lyft}{}{Red,194,39,36;Sunset,202,53,33;Orange,205,68,20;Amber,200,117,42;Yellow,242,169,52;Citron,186,188,44;Lime,112,159,33;Green,56,139,31;Mint,45,118,56;Teal,52,133,135;Cyan,60,132,202;Blue,55,94,248;Indigo,64,13,247;Purple,115,42,248;Pink,176,25,145;Rose,176,32,75}

\usepackage{cite}

\usepackage{microtype}

\usepackage[american]{babel}

\usepackage{array}
\usepackage{multirow}
\usepackage{booktabs}
\usepackage{makecell} %introduces \thead and \makecell

\ifcsname labelindent\endcsname
\let\labelindent\relax
\fi
\usepackage[inline]{enumitem}

\usepackage{subfig}
\makeatletter
\let\endfigure\end@float
\let\endtable\end@float
\makeatother

\newenvironment{lenumerate}[2][]
{\begin{enumerate}[label=(#2\arabic*),leftmargin=0.2in,itemindent=0.15in,#1]}
{\end{enumerate}}

\setlist*[enumerate,1]{label={\itshape\arabic*)}}

\makeatletter
\newcommand{\paragraphswithstop}{%
\let\copyparagraph\paragraph%
\renewcommand\paragraph[1]{\copyparagraph{##1.}}%
}
\makeatother

\usepackage[framemethod=tikz]{mdframed}

\makeatletter
\def\namedlabel#1#2{\begingroup
  #2%
  \def\@currentlabel{#2}%
  \phantomsection\label{#1}\endgroup
}
\makeatother

\makeatletter
\def\namedlabelphantom#1#2{\begingroup
  \def\@currentlabel{#2}%
  \phantomsection\label{#1}\endgroup
}
\makeatother

\newcommand{\parunskip}{\bgroup\unskip\parfillskip=0pt \par\egroup}

\input{preamble/math}
\newcommand{\real}[1]{\mathbb{R}^{#1}{}}

\newcommand{\smallbmat}[1]{\left[\begin{smallmatrix}#1\end{smallmatrix}\right]}

\newcommand{\transpose}{^\mathrm{T}}

\DeclarePairedDelimiter{\norm}{\lVert}{\rVert}

\newcommand{\de}{\mathrm{d}}

\DeclareMathOperator*{\argmax}{\arg\!\max}

\newcommand{\subjectto}{\textrm{subject to}\;}

\providecommand{\cB}{\mathcal{B}}

\providecommand{\cK}{\mathcal{K}}

\providecommand{\cP}{\mathcal{P}}
\providecommand{\cQ}{\mathcal{Q}}

\usepackage{units}

  \newcommand{\newcolorlabel}[2]{%
  \expandafter\newcommand\csname #1\endcsname[1]{%
    \tikz[baseline]{\node[text=white,fill=#2,anchor=base,text height=1.3ex,text depth=0.1ex,font=\sffamily\bfseries]{##1}}}%
}

\newcommand{\newcommenter}[2]{%
  \expandafter\newcommand\csname #1\endcsname[1]{%
    \fcolorbox{#2}{#2}{\color{white}\textsf{\textbf{#1}}}
    {\color{#2}##1}}%
  \expandafter\newcommand\csname at#1\endcsname{%
    \fcolorbox{#2}{#2}{\color{white}\textsf{\textbf{@#1}}}
    {\color{#2}}}%
  \expandafter\newcommand\csname #1cite\endcsname[1]{%
    \csname #1\endcsname {[##1]}
  }%
  \expandafter\newcommand\csname #1ref\endcsname[1]{%
    \csname #1\endcsname {$\blacktriangleright$##1}
  }%
  \expandafter\newcommand\csname #1hl\endcsname[2]{%
    \colorbox{#2}{\color{white}\textsf{\textbf{#1}}}\sethlcolor{Azure2}\hl{##2}~%
    \expandafter\ifx\csname commentarrow\endcsname\relax$\leftarrow$\else \commentarrow[#2]\fi~%
    {\color{#2}##1}}%
  \expandafter\newcommand\csname #1st\endcsname[2]{%
    \colorbox{#2}{\color{white}\textsf{\textbf{#1}}}\sout{##2}~%
    \expandafter\ifx\csname commentarrow\endcsname\relax$\leftarrow$\else \commentarrow[#2]\fi~%
    {\color{#2}##1}}%
}
\newcommenter{TODO}{DodgerBlue1}
\newcommenter{rtron}{Green3}

\usepackage{comment}

\usepackage{pdfcomment}

\usepackage{soul}

\usepackage[normalem]{ulem}

\usepackage{csquotes}

\usepackage{suffix}

\usepackage{environ}

\makeatletter
\newsavebox{\boxifnotempty}
\newcommand{\displayifnotempty}[3]{\sbox\boxifnotempty{#2}\setbox0=\hbox{\usebox{\boxifnotempty}\unskip}%
  \ifdim\wd0=0pt
  \else
  #1\usebox{\boxifnotempty}#3%
  \fi%
}

\newcommand{\ifempty}[2]{\setbox0=\hbox{#1\unskip}%
  \ifdim\wd0=0pt%
  #2%
  \fi%
}

\newcommand{\ifnotempty}[2]{\setbox0=\hbox{#1\unskip}%
  \ifdim\wd0>0pt%
  #2%
  \fi%
}
\makeatother

\newcommand{\switchifempty}[3]{\sbox\boxifnotempty{#1}\setbox0=\hbox{\usebox{\boxifnotempty}\unskip}%
  \ifdim\wd0=0pt{}%
  #2%
  \else{}%
  #3%
  \usebox{\boxifnotempty}%
  \fi%
}

\makeatletter
\@ifundefined{chapter}{\usepackage{algorithm}}{\usepackage[chapter]{algorithm}}
\makeatother
\usepackage{algorithmicx}
\usepackage{algpseudocode}
\makeatother%

\usepackage{scrlfile}

\makeatletter
\newcommand*\newstoreddef[1]{
  \BeforeClosingMainAux{%
    \immediate\write\@auxout{%
      \string\restoredef{#1}{\csname #1\endcsname}%
    }%
  }%
}
\newcommand*{\restoredef}[2]{% used at the aux file
  \expandafter\gdef\csname stored@#1\endcsname{#2}%
}
\newcommand*{\storeddef}[1]{
  \@ifundefined{stored@#1}{0}{\csname stored@#1\endcsname}%
}
\makeatother

\usepackage{pageslts}
\NewEnviron{tee}{\BODY\typeout{Marker Tee [start] ^^J \BODY ^^JMaker Tee [end]}}

\usepackage{cleveref}

\usepackage{tikz}
\usetikzlibrary{calc}
\usetikzlibrary{matrix}
\usetikzlibrary{chains,scopes}
\usetikzlibrary{shapes.geometric}
\usetikzlibrary{arrows.meta}
\usetikzlibrary{decorations.markings}
\usetikzlibrary{decorations.pathreplacing}
\usetikzlibrary{backgrounds}

\tikzset{
  dim above/.style={to path={\pgfextra{
        \pgfinterruptpath
        \draw[>=latex,|->|] let
        \p1=($(\tikztostart)!1.5em!90:(\tikztotarget)$),
        \p2=($(\tikztotarget)!1.5em!-90:(\tikztostart)$)
        in(\p1) -- (\p2) node[pos=.5,sloped,above]{#1};
        \endpgfinterruptpath
      }
    }
  },
  dim double above/.style={to path={\pgfextra{
        \pgfinterruptpath
        \draw[>=latex,|->|] let
        \p1=($(\tikztostart)!3em!90:(\tikztotarget)$),
        \p2=($(\tikztotarget)!3em!-90:(\tikztostart)$)
        in(\p1) -- (\p2) node[pos=.5,sloped,above]{#1};
        \endpgfinterruptpath
      }
    }
  },
  dim below/.style={to path={\pgfextra{
        \pgfinterruptpath
        \draw[>=latex,|->|] let 
        \p1=($(\tikztostart)!-1em!-90:(\tikztotarget)$),
        \p2=($(\tikztotarget)!-1em!90:(\tikztostart)$)
        in (\p1) -- (\p2) node[pos=.5,sloped,below]{#1};
        \endpgfinterruptpath
      }
    }
  },
}

\tikzset{
    right angle quadrant/.code={
        \pgfmathsetmacro\quadranta{{1,1,-1,-1}[#1-1]}     % Arrays for selecting quadrant
        \pgfmathsetmacro\quadrantb{{1,-1,-1,1}[#1-1]}},
    right angle quadrant=1, % Make sure it is set, even if not called explicitly
    right angle length/.code={\def\rightanglelength{#1}},   % Length of symbol
    right angle length=2ex, % Make sure it is set...
    right angle symbol/.style n args={3}{
        insert path={
            let \p0 = ($(#1)!(#3)!(#2)$) in     % Intersection
                let \p1 = ($(\p0)!\quadranta*\rightanglelength!(#3)$), % Point on base line
                \p2 = ($(\p0)!\quadrantb*\rightanglelength!(#2)$) in % Point on perpendicular line
                let \p3 = ($(\p1)+(\p2)-(\p0)$) in  % Corner point of symbol
            (\p1) -- (\p3) -- (\p2)
        }
    }
}

\newcommand{\pgfextractangle}[3]{%
    \pgfmathanglebetweenpoints{\pgfpointanchor{#2}{center}}
                              {\pgfpointanchor{#3}{center}}
    \global\let#1\pgfmathresult  
}

\usetikzlibrary{shapes.arrows}
\newcommand{\commentarrow}[1][Azure4]{\tikz[baseline=-3pt]{\node[shape border uses incircle, fill=#1,rotate=180,single arrow, inner sep=1pt, minimum size=6pt, single arrow head extend=2pt]{};}}

\tikzset{ax/.style={-latex,line width=2pt}}

\tikzset{camera/.style={fill=Sienna1,fill opacity=0.5},%
image plane/.style={draw=RoyalBlue3,line width=2pt}}

\newcommand{\rrtstar}{$\mathtt{RRT^*}$}

\newcommenter{todo}{Firebrick1}
\newcommenter{idris}{Firebrick2}
\title{\LARGE \bf
Navigating the Noise: A CBF Approach for Nonlinear Control with Integral Constraints
}

\author{Idris Seidu and Roberto Tron% <-this % stops a space
\thanks{This work was supported by a grant from NASA.
Idris Seidu and Roberto Tron are with the Department of Mechanical Engineering,
        Boston University, 110 Cummington Mall, MA 02215, United States
        {\tt\small tron@bu.edu, idriseid@bu.edu}}}%

\begin{document}

\maketitle
\thispagestyle{empty}
\pagestyle{empty}

%%%%%%%%%%%%%%%%%%%%%%%%%%%%%%%%%%%%%%%%%%%%%%%%%%%%%%%%%%%%%%%%%%%%%%%%%%%%%%%%
\begin{abstract}
Many physical phenomena involving mobile agents involve time-varying scalar fields, e.g., quadrotors that emit noise.
As a consequence, agents can influence and can be influenced by various environmental factors such as noise. This paper delves into the challenges of controlling such agents, focusing on  scenarios where we would like to  prevent excessive accumulation of some quantity over select regions and extended trajectories. We use quadrotors that emit noise as a primary example, to regulate the trajectory of such agents in the presence of obstacles and noise emitted by the aerial vehicles themselves. First, we consider constraints that are defined over accumulated quantities, i.e functionals of the entire trajectory, as opposed to those that depend solely on the current state as in traditional Higher order Control Barrier Functions (HOCBF).
Second, we propose a method to extend constraints from individual points to lines and sets by using efficient over-approximations.
The efficacy of the implemented strategies is verified using simulations.%, contributing to the understanding of agent navigation in dynamic, noisy conditions. %The efficacy of the implemented strategies is verified using simulations, contributing to the understanding of agent navigation in dynamic, noisy conditions
 Although we use quadrotors as an example, the same principles can equally apply to other scenarios, such as light emission microscopy or vehicle pollution dispersion.
The technical contribution of this paper is twofold.
\end{abstract}

%%%%%%%%%%%%%%%%%%%%%%%%%%%%%%%%%%%%%%%%%%%%%%%%%%%%%%%%%%%%%%%%%%%%%%%%%%%%%%%%

%Outline
%- Introduction:
%  - Motivation
%. - Prior work
%  - Paper contributions
%- Preliminaries
%  - High-Order CBFs (enough so that people can understand how you apply them)
%  - Application to obstacle avoidance for quadrotor with simple-integrator dynamics
%    - Dynamical model ($\dot{x}=u$, shape of the quadrotor)
%    - Reference controller
%    - CBF for obstacle avoidance
%      Define $h(x)$
%      Derive the CBF constraints and the QP
%- CBFs over time integrals
%  Define $J$
%  - Derivation of CBF constraints for general kernels and pointwise cost functional
%    CBF function using $J$
%    Example with parabolic kernel
%  - Derivation with bound for cost on a line
%    CBF function using $J$ and $\max$
%    Explain why discretization is not sufficient
%    Derivation of bound
%    Example with parabolic kernel
%    Discussion for sets in 2D
%    (?) Discussion for sets in higher dimensions
%- Simulations
%  Compare:
%  - CBF only for obstacles
%  - Noise CBF with discretization
%  - Noise CBF with upper bound
%  - Noise CBF with upper bound and different parameters

\section{INTRODUCTION}

%In today's rapidly advancing technological landscape, scalar field-based applications are gaining prominence, influencing areas from environmental monitoring to personal urban mobility. Among the various agents navigating these scalar fields,
Quadrotors have become a transformative force in multiple sectors, including transportation, surveillance, and aerial mapping, revolutionizing the way we approach these fields. However, one critical challenge facing quadrotors, and serving as a potential barrier to their broader acceptance in urban environments, is the noise they emit.

In order to overcome these challenges, this work explores the application of Control Barrier Functions (CBFs) for the navigation and control of agents (quadrotor). CBFs have proven to be an effective tool in the design of controllers for real-time collision avoidance with safety guarantees in nonlinear systems \cite{ames2014control}.

The foundational work by \cite{ames2014control} demonstrates how CBFs can be combined with Control Lyapunov Functions (CLFs) using Quadratic Programs (QPs), providing a robust framework for ensuring the forward-invariance of a set with conditions that are linear in the inputs, hence facilitating their use as QPs.
CBFs find common application in safety-critical systems, particularly when used alongside CLFs (e.g., adaptive cruise control scenarios \cite{ames2017control,ames2014control}). Beyond this, CBFs have been applied to multi-robot systems, illustrating their versatility. For instance,  \cite{borrmann2015control} and subsequent extensions in \cite{Wang:ACC16,wang2017safety} have showcased the application of safety barrier certificates in ensuring collision-free interactions among robots. %reflecting the broad applicability of CBFs in managing complex, dynamic systems.

In another line of work for safety-critical control of quadrotors, CBFs have been integrated with geometric control \cite{wu2016safety,wu2015safety}. %This approach has not only contributed to the practical control of quadrotors but also to the theoretical framework for synthesizing controls that ensure safety while maintaining the system's dynamic performance.
These works, as the majority of the literature, have considered CBFs that are defined on the value of a single, current state of the system. In the type of applications considered in this paper, however, we are interested in constraints that depend on the \emph{entire} trajectory of the system (including past states).

Another significant body of work has focused on the application of CBFs over time-varying sets. For instance, \cite{issei2023safety} introduces a time-varying CBF for nonautonomous control-affine systems, addressing time-dependent safety constraint problems and introducing a control law assisted by humans. Again, however, the applicability of such approaches is limited in scenarios where the definition of the set does not depend on past states of the system.

In parallel, \cite{lindemann2019control,yang2020continuous-time} proposed frameworks and planners based on Signal Temporal Logic and time-varying CBFs. These contributions are crucial for the development of computationally efficient control methods under temporal logic tasks, especially in multi-robot systems. However, while these works discuss spatio-temporal constraints over entire trajectories, they do not address the type of constraints based on the accumulation of a scalar field at specific locations.

%The use of Control Barrier Functions in quadrotor control presents a promising avenue for ensuring, obstacle avoidance, and noise mitigation. By leveraging these advanced control techniques, quadrotors can be deployed in more complex and dynamic environments, expanding their potential applications and enhancing their operational safety. In-depth mathematical analysis of these control strategies and their application will be covered in the sections that follow.

%The key contributions of this paper are centered around innovative techniques for trajectory management of agents, specifically quadrotors, within dynamic environments that involve obstacles and fluctuating noise levels, with adherence to noise constraints.

With respect to the state of the art reviewed above, our paper provides the following main contributions.

\begin{itemize}
    \item %\textbf{Deployment of Integral Cost Functionals in HOCBF:}
    We introduce the use of constraints based on integral cost functionals that track an accumulated cost $J(x,t)$ at specific locations $x$ over time $t$. While we rely on existing HOCBF theory, its application to constraints of the form $J(x,t)\leq J_{\text{limit}}$, where $J$ is a trajectory-dependent integral, is novel.
    \item We introduce a method that saves on computational effort by over-approximating constraints across lines and dense sets, rather than individual points. By setting a Control Barrier Function (CBF) on an upper bound $\bar{J}>J_{\text{max}}$, where $J_{\max}$ is the theoretical maximum value of $J$ over a set, we avoid the need to constantly track \(J_{\text{max}}\), significantly reducing computational demands.
\end{itemize}
Together, these contribution allow us to control mobile agents such as quadrotors on paths that are not only void of collisions, but also enforce limits on the accumulation of a scalar field such as noise. It's important to note that while this paper primarily focuses on quadrotors as an example of mobile agents emitting noise, the underlying principles of our approach have broader applications. These principles can be equally applied to  scenarios such as managing light emission in microscopy to minimize sample bleaching, or controlling vehicle pollution dispersion in an urban environments.

\section{PRELIMINARIES}

Consider the following nonlinear control-affine system:
\begin{equation}
\dot{x} = f(x) + g(x)u
\label{eq:non-linear system}
\end{equation}
where x  $\in \mathbb{R}^n$ is the system state and u $ \in \mathbb{R}^m$ is the control input while $f(x)$ and $g(x)$  are the smooth vector fields.
The $\nabla$ operator represents the gradient for scalar-valued functions that are differentiable with respect to $x$. Additionally, the time derivative of a function $h\bigl(x(t)\bigr)$ with respect to time $t$ is denoted as $\dot{h}(x) = \frac{d}{dt}h(x(t))$. The Lie derivative $L_f h = \nabla h^T f$ measure the changes in a function $h$ along the vector field(s) $f$.

%\begin{definition}[Time-varying Control Barrier Function (CBF)]
%Control Barrier Functions (CBFs) are defined with respect to a region in the state space. For the control affine system \eqref{eq:non-linear system}, suppose we have a continuously differentiable function $h : [0, \infty) \times \mathbb{R}^n \rightarrow \mathbb{R}$ with its super level set $C_t = \{ x \in \mathbb{R}^n : h(t, x) \geq 0 \}$. If this set admits a non-empty interior at each time in $[0, \infty)$, then a smooth function $B : [0, \infty) \times \mathbb{R}^n \rightarrow \mathbb{R} \cup \{ \pm\infty \}$ is called a CBF of $C_t$ if there exist two class $\mathcal{K}$ functions $\alpha_1, \alpha_2$ and $\mu > 0$ such that
%\begin{equation}
%    \frac{1}{\alpha_1(h(x))} \leq B(t, x) \leq \frac{1}{\alpha_2(h(x))},
%\end{equation}
%\begin{equation}
%    \inf_{u \in \mathbb{R}^m} \left\{ \frac{\partial B}{\partial t} + L_f B + L_g B u - \frac{\mu}{B} \right\} \leq 0,
%\end{equation}
%holds for any $t \in [0, \infty)$ and any $x \in \mathring{C}_t$ which is the interior of $C_t$.
%\end{definition}

\begin{definition}[Time-Varying Control Barrier Functions]
A function $h(x,t)$ is a Time-Varying Control Barrier Function (TV-CBF) for the system \eqref{eq:non-linear system} if, for all $x\in\mathbb{R}^n$ and $t\geq0$, it satisfies:
\begin{equation}
    \sup_{u \in \mathbb{R}^m} \left[ L_f h(x, t) + L_g h(x, t)u + \frac{\partial h}{\partial t}(x, t) \right] \geq -\alpha(h(x, t))
\label{eq:CBF_equation}
\end{equation}
\end{definition}
A TV-CBF reduces to a regular CBF if $h$ is constant with respect to time.

The concept of Higher Order Control Barrier Functions (HOCBFs) extends traditional barrier functions to accommodate systems where safety constraints are dependent on higher-order derivatives of the system state. This extension is necessary for systems where the control does not directly influence lower-order derivatives. This concept is made more rigorous with the following.
\begin{definition}[Relative degree]
The relative degree of a sufficiently differentiable function \( h : \mathbb{R}^n \to \mathbb{R} \) with respect to the dynamics \eqref{eq:non-linear system} is defined as the number of Lie derivatives needed until the control input \( u \) explicitly appears.
\end{definition}

\begin{definition} 
[Higher Order Barrier Functions]
For a function \(h : \mathbb{R}^n \times [t_0, \infty) \rightarrow \mathbb{R}\) that is differentiable up to order \(m\), we construct a sequence of functionals
\(\Phi_0, \Phi_1, \ldots, \Phi_m\), where \(\Phi_i : \mathbb{R}^n \times [t_0, \infty) \rightarrow \mathbb{R}\) for \(i \in \{0, \ldots, m\}\), specified by:
\begin{align}
\Phi_0(x,t) &:= h(x,t), \\
\Phi_i(x,t) &:= {\dot \Phi }_{i - 1}(x,t) + \alpha_i(\Phi_{i-1}(x,t)), \quad \text{for } i=1,\ldots,m,
\end{align}
with each \(\alpha_i\) being a class \(\cK\) function.

Correspondingly, we associate a collection of sets \(B_i(t)\) for \(i \in \{1, \ldots, m\}\), defined as:
\begin{equation}\label{set_C}
B_i(t) := \left\{ x \in \mathbb{R}^n : \Phi_{i-1}(x, t) \geq 0 \right\}.
\end{equation}
\end{definition}

\begin{definition}
[Higher Order Control Barrier Functions]
Consider the sets \(B_1(t), \ldots, B_m(t)\) established by the preceding definition, along with their associated functions \(\Phi_0, \ldots, \Phi_m\). A function \(h : \mathbb{R}^n \times [t_0, \infty) \rightarrow \mathbb{R}\) is a High-Order Control Barrier Function (HOCBF) with relative degree \(m\) for the given system \ref{eq:non-linear system} if there exist differentiable class \(K\) functions \(\alpha_1, \ldots, \alpha_m\) and control input $u$ that satisfy:
\begin{equation}
\begin{aligned}
L_f^m h(x, t) + L_g L_f^{m - 1} h(x, t){u} + \frac{\partial^m h(x, t)}{\partial t^m} \\
+ G(h(x, t)) + \alpha_m (\Phi_{m - 1}(x, t)) \geq 0,
\end{aligned}
\label{eq:HOCBF_equation}
\end{equation}
\[
\text{for all } (x,t) \in B_1(t) \cap B_2(t) \cap \ldots \cap B_m(t) \times [t_0, \infty)
\]. The equation above involves \(G(h(x, t))\) which represents the remaining Lie derivatives along $f$ and partial derivatives with respect to $t$ that have a degree equal to or lower than \(m - 1\) as discussed in \cite{xiao2019control}.
\end{definition}

\begin{theorem}\label{thm:HOCBF}
From the HOCBF given in Definition 4 with its related sets \(B_1(t)\), \(B_2(t)\), \ldots, \(B_m(t)\) defined in \eqref{set_C}, if \(x(t_0) \in B_1(t_0) \cap B_2(t_0) \cap \ldots \cap B_m(t_0)\), then any Lipschitz continuous controller \(u(t) \in \mathcal{U}\) that satisfies \eqref{eq:HOCBF_equation} for all \(t \geq t_0\) renders the sets \(B_1(t)\), \(B_2(t)\), \ldots, \(B_m(t)\) forward invariant for system \eqref{eq:non-linear system}.
\end{theorem}

\section{CBF OVER TIME INTEGRALS}
%In the preceding sections, we established the motivation and framework for employing CBFs to manage the trajectory of quadrotors, specifically addressing the challenge of noise accumulation. 
In this section, we extend the traditional CBF formulation to encompass constraints that are based on the accumulated effect over time, rather than relying solely on instantaneous state values.
Let $x(t)$ be the trajectory of an agent, and $p(q;x)$ be a scalar field function $\real{d}\to \real{}_+$ representing the  effect of the agent at each point $q$ when the agent is at location $x$.
We define an integral cost functional of the form:
\begin{equation}
J(q,t)=\int_0^t p\bigl(q,x(\tau)\bigr) \de \tau,
\label{eq:J(x,t) integral}
\end{equation}
which is a functional representing the cumulative effect at a given location $q$.

\begin{problem}\label{problem:statement}
Given a set $\mathcal{Q} \subseteq \mathbb{R}^d$, let $J_{\text{limit}}$ be the maximum permissible value of $J(q,t)$. Find constraints on the control input $u$ for the dynamical system \eqref{eq:non-linear system} such that the resulting trajectory $x(t)$ satisfies $J(q,t)\leq J_{\text{limit}}$ for every time $t>0$ and every location $q \in \mathcal{Q}$.
\end{problem}

We propose to tackle Problem \ref{problem:statement} using the following function as a CBF
\begin{equation}\label{eq:hlimit}
    h_{J}(x,t)=\min_{q\in\cQ} \bigl(J_{\textrm{limit}}-J(q,t)\bigr).
\end{equation}
In the following we derive constraints for different types of sets $\cQ$: single points, line segments, and general polygons. Our analysis primarily focuses on two-dimensional (2-D) spaces, but the same theory could be extended to three-dimensional (3-D) spaces, although we leave this generalization for future research.

\subsection{Constraint for a single point}
If the set $\cQ$ contains a single point, $\cQ=\{q_0\}$, the Control Barrier Function \eqref{eq:hlimit} reduces to:
\begin{equation}
h_{J}(x,t;q_0)=J_{\text{{limit}}}-J(q_0,t).
\end{equation}
Taking derivatives until the control $u$ appears explicitly we obtain:
%then the CBF becomes $\dot{h_{eff}}(x)+\alpha h_{eff}(x)\geq 0$ where:
\begin{align}
\dot{h}_{J}(x,t;q_0) &= -p\bigl(q_0,x\bigr),\\
\ddot{h}_{J}(x,t;q_0) &= -\nabla_x p\bigl(q_0,x\bigr)\transpose (f(x)+g(x)u);
\end{align}
here and for the remainder of the paper, we assume that $\nabla_x p\bigl(q_0,x\bigr)\transpose g(x)\neq 0$, which implies that $h_{J}$ has relative degree $m=2$.
%it is observed that the control input, denoted as $u$, is absent from the formulation. To reconcile this and ensure the incorporation of the control input in the formulation,
Applying the HOCBF framework reviewed above, we have:

\begin{align}
\Phi_{J,1}(x,t;q_0) &:= -p\bigl(q,x(t)\bigr) + \alpha_1 h_{J}(x,t)\\
\Phi_{J,2}(x,t;q_0)&:=\dot{\Phi}_{J,1}(x,t) + \alpha_2 \Phi_{J,1}(x,t)
\end{align}

The constraint then becomes:
\begin{equation}
    \Phi_{J,2}(x,t;q_0)\geq 0
\end{equation}

This constraint is linear in $u$, and will be incorporated in a Quadratic Program (QP). Overall, this case is a relatively straightfoward application of the HOCBF framework, with the only consideration being that the integral in the cumulative cost \eqref{eq:J(x,t) integral} makes the use of high-order CBFs necessary (i.e., the relative degree is $m\geq 2$) even when the dynamics \eqref{eq:non-linear system} is first-order.

\subsection{Constraint for a line segment using upper bounds}\label{sec:line bound}
This section considers the case where the set $\cQ$ is a line segment with endpoints $q_0, q_K$. We use the following parametrization of the set, with $s\in[0,1]$:
\begin{equation}
    q(s)=(1-s) q_0 + s q_K;
\label{eq:segment equation}
\end{equation}
Ideally, we would like to satisfy the constraint on $\Phi_{J2}$ for all the points in the set (i.e., all the values of $s$). However, this leads to having an infinite number of constraints because the parameter $s$, which dictates the point's location on the line segment, is a continuous value. 
An alternative strategy is to keep track of $s^*$ and $J_{\max}$ defined by the equation:
\begin{align}
s^*(t)&=\argmax_s J(q(s),t)\\
J_{\max}(t;q_0,q_1)&= J\Bigl(q\bigl(s^*(t)\bigr),t\Bigr)\label{eq:Jmax(t;x_0}
\end{align}

Note that we employ the $\max$ operator (instead of the minimum operator) due to the presence of a negative sign in the minimization of \eqref{eq:hlimit}. 

However, this strategy would require updating $s^*$, which also requires keeping track of the full function $J(q(s),t)$ (which is infinite-dimensional, since it is a continuous function).

In practice, it is necessary to introduce some form of appoximation.
%From these perspectives, two main methodologies emerge. The first method involves discretization of the segment, transforming the infinite continuum of $s$ values into a finite set of points, but it has its limitation. Secondly, instead of tracking where $s^*$ is, we keep track of a bound on where $J$ at the maximum would be.
We consider two strategies: a na\"ive approximation of the segment with points, and an approximation of the CBF constraint with an upper bound that can be easily updated.

\subsubsection{Na\"ive solution using discretization}
Let \(\bar{\cQ}=\{q_k\}_{k=0}^{K}\) be a discretization of the original set \(\cQ\) with $K$ points. The most na\"ive way to approximate \eqref{eq:hlimit} is to transform it into a series of constraints \(h(x,t;q_k)\geq 0\) for every $k$ in $\{0,\ldots,K\}$, where each point $q_k$ is an element of $\bar{\cQ}$.
We then define a HOCBF constraint
\begin{equation}\label{eq:discretization}
    \Phi_{J,2}(x,t;q_k)\geq 0
\end{equation}
for every $k$.

Effectively, this replaces $\min_{q\in\cQ} J(q,t)$ in \eqref{eq:hlimit} in  with $\min_{q\in\bar{\cQ}} J(q,t)$. This, however, causes two problems: in order to obtain a good approximation, we might need a large number of points (thus incurring in the curse of dimensionality). More importantly, there might always be a point $q\in\cQ$ for which $J(q,t)>J_{\textrm{limit}}$, independently from how fine the discretization is (i.e., independently from the value of $K$). This inherent deficiency prompts the next approach.

\subsubsection{Approximation without discretization}
Rather than attempting to track the maximum cumulative effect over  $s$ directly, we propose to compute a bound. %This approach overcomes the challenges associated with maintaining an infinite number of constraints and the computational overhead of constantly updating $s^*$ to reflect the point of maximum effect along a line segment.
Let $q_0$, $q_K$ be two consecutive points on the boundary of the region we want to protect from the cumulative noise. We would like to enforce a CBF constraint on $J_{\max}$ in \eqref{eq:Jmax(t;x_0}, as that would protect all the points on the line between $q_0$ and $q_K$. However, keeping track of $J_{\max}$ is computationally costly because we would need to keep track of $s^*$, which in turn would require a maximization of $J(q(s),t)$ (which would require evaluating integrals for the sequence of $s$ decided by the solver, for every time step).

Instead, we define an upper bound $\bar{J}$ on $J_{\max}$ as:
\begin{align}
    \bar{J}(0;q_0,q_K)&=J_{\max}(0) \label{eq:J_bar1}\\
    \dot{\bar{J}}(t;q_0,q_K)&=\max_s p(q(s),x(t)).\label{eq:J_bar2}
\end{align}
Note that the argument of the maximum in \eqref{eq:J_bar2} is the kernel $p$, not the integral over time. Intuitively, we propose to use $\bar{J}$ to define a lower bound on \eqref{eq:hlimit}, and then use this lower bound as a CBF we can enforce safety for every point.
\begin{proposition} Define the \emph{lower-bound} CBF
\begin{equation}
h_{\bar{J}}(x,t;q_0,q_K) =\min_{q\in\cQ}\bigl(J_{limit}-\bar{J}(t;q_0,q_K)\bigr),
\end{equation}
and the corresponding HOCBFs functions
\begin{align}
\bar{\Phi}_1(x,t;q_0,q_K) &:= -p\bigl(q,x(t)\bigr) + \alpha_1 h_{\bar{J}}(x,t),\\
\bar{\Phi}_2(x,t;q_0,q_K)&:=\dot{\bar{\Phi}}_1(x,t) + \alpha_2 \bar{\Phi}_1(x,t).
\end{align}
If the control $u(t)$ satisfies the HOCBF constraint
\begin{equation}
    \bar{\Phi}_2(x,t;q_0,q_K) \geq 0 \label{eq:phi2 Jbar}
\end{equation}
then $J_{\max}(t)\leq J_{limit}$ for all $t\geq 0$.
\end{proposition}
\begin{proof}
First, from the definition of $J_{\max}$ in \eqref{eq:Jmax(t;x_0} and the fundamental theorem of calculus,
\begin{equation}
    \dot{J}_{\max}(t)=p(q(s^*),x(t)).
\label{eq:Jmax_dot}
\end{equation}

From the definition of $\bar{J}$ in \eqref{eq:J_bar2}, we have
\begin{equation}
    \dot{\bar{J}}(t)=\max_s
    \dot{J}(t)=\max_s p(q(s),x(t))\geq \dot{J}_{\max}(t),
\label{eq:dot bar J}
\end{equation}
where the last inequality is given by the fact that the location $s^*$ where $J$ is maximum is not necessarily also the one that is currently increasing the most, i.e., where $\dot{J}$ is maximum.
%Since from \eqref{eq:Jmax_dot} and \eqref{eq:J_bar2}, \(\dot{J}_{\text{max}}(t)\) is the rate of change of \(J_{\text{max}}\) and it's determined by the value of \(p\) at a specific optimal point \(s^*\) (the one that maximizes \(J\)) while \(\dot{\bar{J}}(t)\) is the rate of change of \(\bar{J}\) and it's determined by the maximum possible value of \(p\) across all points \(s\).

From \ref{eq:J_bar1} and \eqref{eq:dot bar J}, and applying Gronwall's comparison lemma \cite{gronwall1919note} we have \eqref{eq:phi2 Jbar}.

\begin{equation}
J_{\max}(t) \leq \bar{J}(t)
\end{equation}
for all $t\geq 0$. Then, if \eqref{eq:phi2 Jbar} holds, from \Cref{thm:HOCBF} we have
\begin{equation}
J_{\max}(t)\leq \bar J(t)\leq J_{limit}.
\end{equation}
The claim follows.
\end{proof}

\subsection{Constraints for a general 2-D polygon using upper bounds}
Given a polygon \( \cP \) defined by vertices \(\{v_1, v_2, \ldots, v_N\}\) in \(\mathbb{R}^n\), the boundary of \( P \) can be defined as a sequence of linear segments \([v_i, v_{i+1}]\) for \( i = 1, 2, \ldots, N \), with \( v_{N+1} = v_1 \) to complete the polygon. We assume the use of a regular CBF to enforce $x\notin\cQ$. As a result, we can reduce the goal of protecting all segments in $\cQ$ to the goal of protecting the boundary of \( \cP \). we extend the equation \eqref{eq:hlimit}, where the  CBF $h_{J}$  is adapted to address the cumulative effect constraints across the polygon's boundary segments.
For each segment \([v_i, v_{i+1}]\), we can apply the methodology described in \Cref{sec:line bound} to define the lower-bound CBF $h_{\bar{J}}(x,t;v_i,v_{i+1})$ for each segment $i\in\{1,\ldots,N\}$, which lead to $N$ HOCBF constraints of the form
\begin{equation}\label{eq:Polygon_constraint}
\bar{\Phi}_2(x,t;v_i,v_{i+1}) \geq 0.
\end{equation}
Furthermore, to ensure that the agent doesn't pass through this polygon, the obstacle CBF constraint is applied alongside equation \ref{eq:Polygon_constraint}.
%The polygon obstacle CBF is defined as \(h_{polyobs}(x) = \lVert x_{\text{{quad}}} - x_{\text{{poly}}} \rVert - r_{\text{{quad}}}\), where $x_{\text{{poly}}}$ is the point on the edges of the polygon that is nearest to a propeller disk, specifically the one that is closest among the four available propeller disks to the polygon.
\begin{comment}
\subsubsection{Option 1: Use Lipschitz constant}

First, compute the Lipschitz constant $L$ of $J$. Then use the fact that
\begin{equation}J(x,q)\leq \bar{J}_k(x,q)= J(x_k,q)+L\norm{x-x_q}.
\end{equation}

Then use
\begin{equation}
h_k(q,t)=\min_{x\in\cB(x_k,r_k)} J_{\max}-\bar{J}_k(x,q)
\end{equation}

and use the techniques shown by Max Cohen's paper.
\end{comment}

\section{CASE STUDY}
This section outlines a case study focusing on the application of our CBF approach on a simplified model of a quadrotor with first-order-integrator dynamics navigating through a noise-sensitive environment. We consider a simple isotropic distance-based noise model as the scalar field. The goal is then to guarantee that the accumulated noise at any point on a building (modeled as a polygon) is below the desired limit.

\subsection{Quadrotor model}
In our representation, the quadrotor is illustrated as comprising four propellers, represented by four disks, tangent to each other, each with its own center. This arrangement surrounds the quadrotor's central point, which we denote as $x$. These individual circles are referred to as ``propeller disks'' , and $r_{\text{quad}}$ is the radius of each disk.

\subsection{Obstacle model}
The obstacle is represented as a static box and for the purpose of collision avoidance, the quadrotor's proximity to the obstacle is quantified by the position of the nearest propeller disk. Within the context of the quadrotor model, $x_{\text{quad}}$ is defined as the center of the propeller disk that is closest to the box at any given time. Also, $x_{\text{box}}$ is defined as the point on the surface of the box that is nearest to the aforementioned propeller disk. These are used to define the CBF for the obstacle.

\subsection{Application of CBF-QP for obstacles}
%We have presented the foundational concepts of CBF and HOCBF in section 2, and building upon these concepts, in this subsection, we shall delve into a specific application of quadrotor navigation. We illuminate the process wherein a quadrotor navigates its surroundings, effectively bypassing obstacles, utilizing the Control Barrier Function–Quadratic Program (CBF-QP).
The dynamics of the quadrotor is modeled as a single integrator $\dot{x} = u$. We define the barrier function, \(h_{obs}(x)\), as \(h_{obs}(x) = \lVert x_{\text{{quad}}} - x_{\text{{box}}} \rVert - r_{\text{{quad}}}\). %From the definition of CBF in \eqref{eq:CBF_equation}, for our system, the barrier function satisfies $ \sup_{u \in \mathbb{R}^m} [\dot{h}(x) + \alpha(h(x))] \geq 0$, and 
The CBF-QP for the quadrotor becomes

\begin{equation}
    \begin{aligned}
        \min_{u} &\norm{u-k_{\textrm{ref}}(x)}^2\\
        \subjectto & \dot{h}_{obs}(x)+\alpha_1 h_{obs}(x)\geq 0
    \end{aligned}
\label{eq:CBF-QP-OBS}
\end{equation}
where $k_{\textrm{ref}}(x)=x_{\textrm{goal}}-x(t)$.

Herein, \(x(t)\) denotes the quadrotor's position at discrete time intervals, while \(x_{\text{{goal}}}\) represents the quadrotor's intended destination. An example of this is shown in the simulation section.  In the case of multiple obstacles, the CBF-QP can be extended with multiple constraints, each representing a different obstacle. For each obstacle, we can define a barrier function $h^{(i)}_{obs}(x)$
where $i$ is the obstacle index, ensuring that the quadrotor maintains a safe distance from all obstacles. The QP can be augmented to accommodate all these constraints, ensuring a collision-free trajectory even in cluttered environments.

\subsection{Parabolic kernel noise model}

We represent the state-dependent noise scalar field $p(q, x(\tau))$ for the quadrotor using a parabolic kernel:
\begin{equation}\label{eq:kernel}
p(q, x(\tau)) = \begin{cases}
    A - \sigma\norm{q-x(\tau)}^2 & \text{if } \norm{q-x(\tau)} \leq \sqrt{\frac{A}{\sigma}} \\
    0 & \text{otherwise}
\end{cases}
\end{equation}

In this formulation, $A$ represents the peak value of the kernel, indicative of the maximum intensity of the noise (which it a the quadrotor's center). The parameter $\sigma$ controls the width of the parabola, essentially determining the spread of the kernel. The right hand side $\sqrt{\frac{A}{\sigma}}$ represents the effective radius within which the kernel possesses non-zero values, marking the boundary of the noise influence. Note that this is a very simplified model, but it captures the fact that the influence of the noise on point $q$ is nonlinear and state-dependent.

In general, the computation of the bound $\bar{J}$ requires the maximization of $p$ (equation \eqref{eq:dot bar J}). For the case of the parabolic kernel \eqref{eq:kernel} and where the set $\cQ$ is a line parametrized by $s$, the optimization problem can be solved in closed form.

\begin{lemma}
Given the parabolic kernel noise model \eqref{eq:kernel} defined over the line segment parameterized by \( s \) with endpoints \( q_0 \) and \( q_K \), the optimization problem is to maximize \( p(q(s), x(\tau)) \). The solution to this optimization problem is:
\begin{equation}\label{eq:sstar}
s^* = \frac{(x(\tau) - q_0)^T (q_K - q_0)}{|q_K - q_0|^2}.
\end{equation}

The maximum corresponds to the point $q^* \in\ q_0, q_K$:
\begin{equation}\label{eq:qstar}
q^* =
\begin{cases}
q_0 & \text{if } s^* < 0 \\
q_K & \text{if } s^* > 0 \\
q(s^*) & \text{otherwise}
\end{cases}
\end{equation}

\end{lemma}
\begin{proof}
We first express \( p(q(s),x(t)) \) in terms of \( s \):
\[ p_s(q(s),x(t)) = A - \sigma \norm{(1-s)q_0 + sq_K - x}^2 \]
subject to:
\[ \norm{{(1-s)q_0 + sq_K - x(t)}} \leq \sqrt{\frac{A}{\sigma}} \]
We differentiate \( p_s \) with respect to \( s \):
\[ \frac{dp_s}{ds} = -2\sigma \bigl((1-s)q_0 + sq_K - x(t)\bigr)\transpose(q_K - q_0) \]
Setting this to zero, we get:
\[ \bigl(q(s)-x(t)\bigr)^T (q_K - q_0) = 0 \]
Solving for $s$:
\begin{multline}
    0=((1-s)q_0 + sq_K - x(t))^T (q_K - q_0) \\
    =(q_0 -x(t) + (q_K -q_0)s)^T (q_K - q_0)\\
    =(q_K -q_0)^T (q_K - q_0)s +(q_0 -x(t))^T (q_K - q_0),
\end{multline}
from which \eqref{eq:sstar} follows.
\end{proof}

The value of \( s^* \) is used to determine the optimal point \( q^* \) on the line segment as shown in \eqref{eq:qstar}.
With $q^*$ identified, we can now define the barrier function $\bar{\Phi}_2(x,t;q_0,q_1)$, which incorporates the noise footprint constraints into the control strategy, ensuring that the cumulative effect at $q^*$ remains within the specified limits.

\begin{multline}
    \bar{\Phi}_2(x,t;q_0,q_K) := -2\sigma (q-x(t)) \, u \\
    + \alpha_2 \Bigl(-\bigl(A - \sigma||q-x(t)||^2\bigr) 
    + \alpha_1 \bigl(J_{\text{limit}} - \bar{J}(q,t)\bigr)\Bigr) \geq 0
\end{multline}

Consequently, from \eqref{eq:CBF-QP-OBS}, the refined Quadratic Programming (QP) formulation is given as:
\begin{equation}\label{eq:qp noise}
\begin{aligned}
\min_{u} & \|u - k_{\textrm{ref}}(x)\|^2 \\
\text{s.t.} \quad
& \dot{h}_{obs}(x)+\alpha_1 h_{obs}(x)\geq 0,\\
& \bar{\Phi}_2(x,t;q_0,q_K)\geq 0.
\end{aligned}
\end{equation}
The solution to \eqref{eq:qp noise} (assumin it is feasible) will respect both the obstacle and the noise footprint constraints.

%In this formulation, CBF represents the control barrier function corresponding to the obstacle, while CBF2 is associated with the noise footprint's control barrier function.

\section{SIMULATION RESULTS}
This section details simulation results demonstrating the quadrotor's navigation using CBFs for obstacle avoidance and noise management. We showcase how these strategies enable the quadrotor to safely navigate and comply with environmental constraints, highlighting the practical effectiveness of our control approach.

\subsection{Simulation parameters}
We evaluate our approach with a simulation in an environment with a single obstacle and the simulation is based on Euler's method for integration through time.The focus of the paper is on local control, and a single obstacle allows us to better demonstrate the effect of the various parameters. Navigating more complex environments would require integration with high-level path planning algorithms such as \rrtstar \cite{karaman2011samplingbased}, which, however, are out of scope for this paper.
%to show how the quadrotor avoid  obstacles (box) using the QP formulation in \ref{eq:CBF-QP-OBS}.
The following parameters were used for the simulation: the class-$\cK$ function for the obstacle CBF and noise CBF constraint uses $\alpha_1 = 3.0$; for the noise CBF constraint, we use $\alpha_2 = 6.0$; the size of the quadrotor is given by $r_{\text{quad}} = 0.1$; the Euler integration time step is $\delta = \unit[0.1]{s}$; the obstacle is a square $\cQ=[0,1]\times[0,1]$; finally, the initial position is $x(0)=\smallbmat{3\\3}$.  As the reference controller, we use a simple proportional controller $k_{\textrm{ref}}(x)=x_{\textrm{goal}}-x$, where $x_{\textrm{goal}}=\smallbmat{-2\\ -1}$.
\subsection{Obstacle CBF alone}
Figure \ref{fig:sub1} shows the result of the simulation using the QP with only the obstacle CBF constraints \eqref{eq:CBF-QP-OBS}. In the figure, the trajectory's goal is symbolized by a green star while the starting position is shown in orange, and it is observable that the quadrotor maintains a safe distance from the obstacle towards the goal, adhering to established safety measures.
Since this simulation does not take into account the constraint $J(t) < J_{\text{limit}}$, while the quadrotor effectively avoids the specified physical obstacles, it may exceed the allowable cumulative impact at certain locations. As seen in \Cref{fig:Jxts} the cumulative noise $J_{obs}$ exceeded the limit $J_{limit}$ for the trajectory. 
%since the noise CBF constraint was not implemented in the QP. 

\begin{figure}
    \centering
    \subfloat[System trajectories for the QP-based control of the quadcopter with box obstacle]{\includegraphics[trim=0cm 4mm 0cm 4mm,clip]{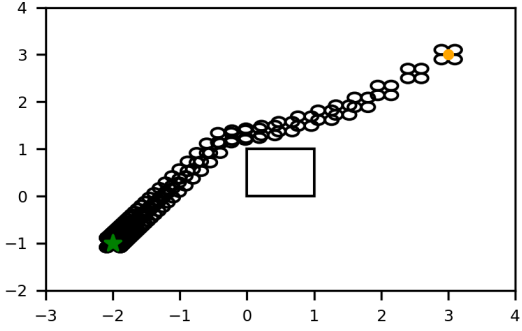}\label{fig:sub1}} \\
    \subfloat[Application of Discretization]{\includegraphics[trim=0cm 4mm 0cm 4mm,clip]{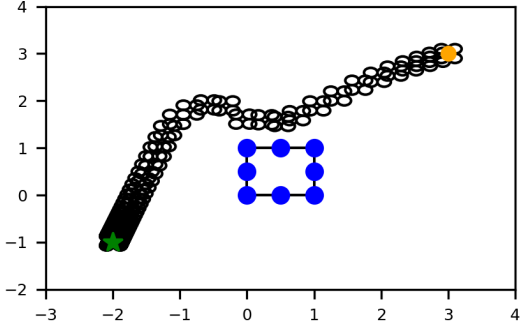}\label{fig:sub2}} \\
    \subfloat[QP-based control of the quadcopter with noise and obstacles using approximation without discretization]{\includegraphics[trim=0cm 4mm 0cm 4mm,clip]{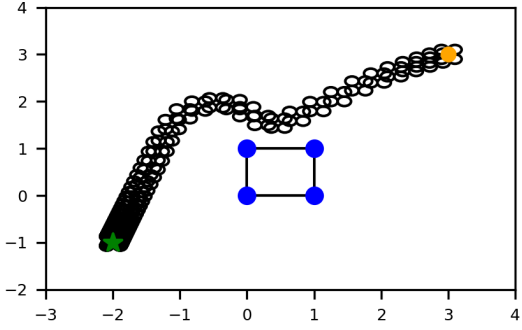}\label{fig:sub3}} \\
    \subfloat[Approximation without discretization with different parameters]{\includegraphics[trim=0cm 4mm 0cm 4mm,clip]{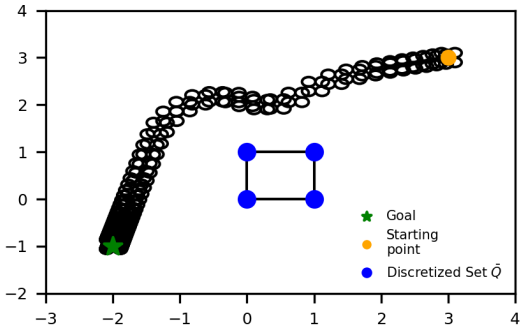}\label{fig:sub4}}
    \caption{System trajectories with just the obstacle, noise using discretization, noise but with approximation without discretization and approximation without discretization with different parameters }
    \label{fig:main}
\end{figure}

\begin{figure}
  \centering
  \includegraphics[trim=0cm 4mm 0cm 0mm,clip]{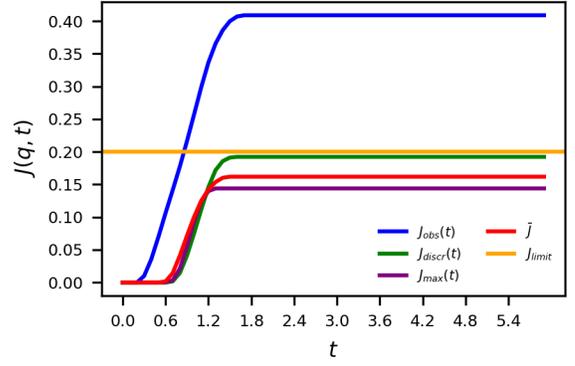}
  \caption{$J_{\textrm{obs}}(t)$, $J_{discr}(t)$, $J_{\max}(t)$, $\bar{J}$ and $J_{limit}$ with time}
  \label{fig:Jxts}
\end{figure}

\subsection{Cumulative noise CBF with discretization}\label{sec:disc_simu}
In this section, we show the simulation of the quadrotor with discretization of sets to account for the noise around the obstacle. For this simulation, we used $J_{\textrm{limit}} = 0.2$, $A = 0.5$, $\sigma = 0.35$ and $K = 8$. \Cref{fig:sub2} shows the result of the simulation, where we use eight constraints obtained by discretizing the boundary of the obstacle and applying individual point-wise constraints as described in \Cref{eq:discretization}. Incorporating the noise CBF constraints in the QP formulation, as visualized in \Cref{fig:sub2}, results in a trajectory that is more considerate of the noise impact of the quadrotor compared to the simulation shown in \Cref{fig:sub1}. This is evident from \Cref{fig:Jxts}, where the cumulative noise  $J_{\textrm{discr}}(t)$, representing the cummulative noise subject to the discretized constraints, remains below $J_{\textrm{limit}}$. Also \Cref{fig:Jxts} shows the  plots of $J_{\max}(t)$ being below the limit, where  $J_{\text{max}}(t) = \max_{i=0}^{99} J(q_i,t)$ and $q_0, q_1, \ldots, q_{99}$  are the discretized points of the square.
\subsection{Cumulative noise CBF with bound \texorpdfstring{$\bar{J}$}{J-bar}}

In this section, we share the results of simulations where we applied a Control Barrier Function (CBF) constraint on $\bar{J}$. The parameters used are the same as above.
% The following parameters were used in the simulations:$J_{limit} = 0.2$, $A = 0.5$, $\sigma = 0.35$, $\alpha_1 = 3.0$, $\alpha_2 = 6.0$, time step($\delta) = 0.1$ and $r_{\text{quad}} = 0.1cm$. We started from the position $[3, 3]^T$ while the goal was at $[-2.0, -1.0]^T$ and the obstacle is a square with $[0,1]\times[0,1]$.
Fig. \ref{fig:sub3}  shows the quadrotor navigating toward its goal, detouring upon encountering an obstacle. The escalating noise around this obstacle prompts the quadrotor to maintain distance, thereby steering clear of the accumulating noise while persistently moving toward its target. The difference between this method and the method in \Cref{sec:disc_simu}, is that in the discretization method depicted in Figure \ref{fig:sub2}, the noise constraints are applied only at specific points along the boundary of the obstacle. This method risks overlooking some areas along the boundary where noise accumulation might exceed the limit, because it does not account for the entirety of the line segment, while in the bound $\bar{J}$ approach, a CBF constraint is imposed on an over-approximation of the maximum cumulative noise, ensuring that no point along the line segment will exceed the noise threshold which can be seen in \Cref{fig:Jxts} as \( \bar{J} < J_{\text{limit}}\). The corners of the box, are points used to symbolize the segments forming the boundary of the obstacle $q_0$, $q_1$, $q_2$, and $q_3$.

\subsection{Other parameters}
We also ran the simulation with different parameters where $A=1.0$, $\sigma=0.55$ and the result can be seen in Fig. \ref{fig:sub4} which shows the quadrotor giving more safe distance away from the obstacle while heading towards the goal. This shows that by using a larger noise radius, the quadrotor adjusts its path to keep a wider distance from the obstacle as the noise radius increases.
Finally, we ran a simulation with  $x_{\textrm{goal}}=\smallbmat{0.5\\1}$ near the obstacle, and from Figure \ref{fig:goal at boundary}, it shows the quadrotor unable to get to the goal due to the accumulated noise at the obstacle which can be also seen in Figure \ref{fig:x and y plot of the goal at abstacle} which shows the position of the robot at each time step and the plot shows the quadrotor bouncing back.

\begin{figure}
    \centering
    \subfloat[Goal at the obstacle]{\includegraphics[trim=0cm 4mm 0cm 0mm,clip]{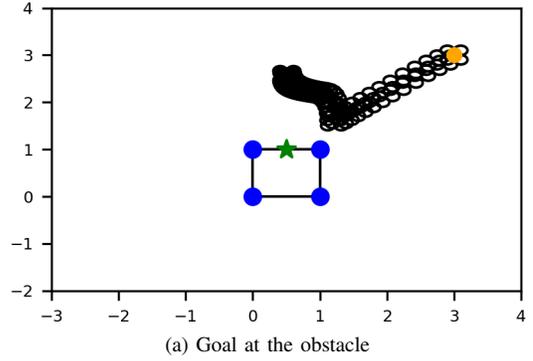}\label{fig:goal at boundary}}

    \subfloat[Position at each time step]{\includegraphics[trim=0cm 4mm 0cm 4mm,clip]{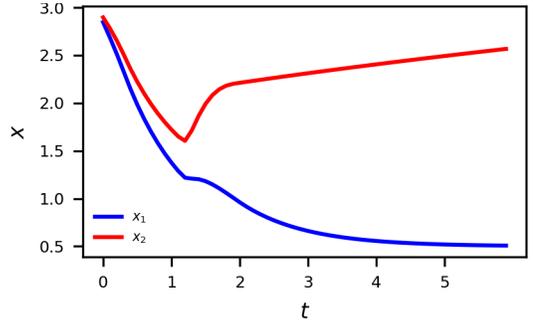}\label{fig:x and y plot of the goal at abstacle}}

    \caption{Trajectory of the quadrotor when the goal is at the boundary of the obstacle}
    \label{fig:main2}
\end{figure}

\section{CONCLUSIONS}
% This paper presents innovative techniques for enhancing quadrotor navigation in complex environments featuring obstacles and noise-sensitive areas.
We presented a method to consider integral cost functionals into the framework of higher-order control barrier functions, enabling constraints over entire trajectories rather than just instantaneous states. Additionally, we present an efficient over-approximation method to handle constraints over large regions. Our simulations validate the ability of the agent (a simplified quadrotor) to navigate safely through environments while mitigating the impact of noise in sensitive areas. %This work provides insights for advanced autonomous navigation and noise management.
While we used acoustic noise emitted from quadrotors as an illustrative application, the theory we presented could be applied to other scenarios, such as managing light emission in microscopy or controlling vehicle pollution dispersion. %, showcasing the broader applicability of the proposed methodologies.
Future research may extend this work to multi-agent systems, consider more accurate noise propagation and quadrotor dynamical models, and combine the our low-level control constraints into complete path planning solutions for cluttered environments.

\bibliographystyle{biblio/ieee}

\bibliography{biblio/IEEEfull,biblio/IEEEConfFull,biblio/OtherFull,% Do not insert spaces in this command, otherwise it will not work.
  references,%
  biblio/tron,%
  biblio/controlCBFs,%
  biblio/formationControl,%
  biblio/websites}

\end{document}